\documentclass[12pt]{amsart}

\usepackage{amssymb}
\usepackage{mathtools}
\usepackage{bm}
\usepackage{graphicx}
\usepackage{psfrag}
\usepackage{color}
\usepackage{url}
\usepackage[linesnumbered, ruled, vlined]{algorithm2e}
\usepackage{algpseudocode}
\usepackage{fancyhdr}
\usepackage{xy}
\usepackage{dirtytalk}
\usepackage{float}
\usepackage{stmaryrd}
\input xy
\xyoption{all}

\newtheorem{theorem}{Theorem}[section]
\newtheorem{proposition}[theorem]{Proposition}
\newtheorem{remark}[theorem]{Remark}

\newtheorem{cor}[theorem]{Corollary}

\theoremstyle{definition}
\newtheorem{definition}[theorem]{Definition}
\newtheorem{example}[theorem]{Example}
\newtheorem{claim}{\bf Claim}
\numberwithin{equation}{section}

\newcommand\nn{\mathbb{N}}
\newcommand\qq{\mathbb{Q}}
\newcommand\rr{\mathbb{R}}

\begin{document}%%%%%%%%%%%%%%%%%%%%%%%%%%%%%%%%%%%%%%%
	
	\mbox{}
	\title{Approximating Length-Based Invariants in Atomic Puiseux Monoids}
	\author{Harold Polo}
	\address{Mathematics Department\\University of Florida\\Gainesville, FL 32611}
	\email{haroldpolo@ufl.edu}
	\date{\today}
	
	\begin{abstract}
		 A numerical monoid is a cofinite additive submonoid of the nonnegative integers, while a Puiseux monoid is an additive submonoid of the nonnegative cone of the rational numbers. Using that a Puiseux monoid is an increasing union of copies of numerical monoids, we prove that some of the factorization invariants of these two classes of monoids are related through a limiting process. This allows us to extend results from numerical to Puiseux monoids. We illustrate the versatility of this technique by recovering various known results about Puiseux monoids.
	\end{abstract}
	
	\maketitle
	
	%%%%%%%%%%%%%%%%%%%%%%%%%%%%%%%%%%%%%%%%%%%%%%%%%%%%%%
	\section{Introduction} \label{sec:intro}%%%%%%%%%%%%%%%%%%%%%%%%%% \label{sec:intro}%%%%%%%%%%%%%%%%%%%%%%%%%%%%%%%%%%%%%%%%%%%
	
	A monoid $M$ is atomic provided that every nonunit element can be represented as a product of finitely many irreducibles. If for each nonunit element of $M$ such a representation is unique, up to permutation, then $M$ is called a unique factorization monoid (UFM). For example, the positive integers with the standard product is a UFM by the Fundamental Theorem of Arithmetic. Factorization theory studies how far is an atomic monoid from being a UFM, and several algebraic invariants has been introduced to quantify this deviation (see \cite{AGQZ} and references therein).
	
	Numerical monoids, that is, cofinite additive submonoids of the nonnegative integers, have been significantly investigated in the context of factorization theory; much of the recent literature has focused on the computational aspects of their factorization invariants (see, for example, \cite{TBCORP2017}). Since numerical monoids are finitely generated, calculating factorization invariants in this setting is highly tractable~\cite{GS2016}. This motivated the implementation of a GAP~\cite{GAP} package, \texttt{numericalsgps}~\cite{MDGSJM}, to assist researchers in the area. Thus, numerical monoids constitute an ideal framework to study factorization invariants.
	
	Additive submonoids of the nonnegative cone of $\qq$, also called Puiseux monoids, are natural generalizations of numerical monoids. A systematic investigation of these monoids started just a few years ago in \cite{GOTTI16} and, consequently, we do not know much about their factorization invariants. The crux of this article is to study the set of lengths (and related factorization invariants) of Puiseux monoids through their representation as increasing unions of copies of numerical monoids.
	
	%%%%%%%%%%%%%%%%%%%%%%%%%%%%%%%%%%%%%%%%%%%%%%%%
	\section{Preliminary} \label{sec:definition of GL}%%%%%%%%%%%%%%%%%
	%%%%%%%%%%%%%%%%%%%%%%%%%%%%%%%%%%%%%%%%%%%%%%%%
	In this section, we introduce the concepts and notation necessary to follow our exposition. General references for factorization theory can be found in \cite{AGHK06}.
	
	Throughout this article, we let $\nn$ and $\nn_0$ denote the set of positive and nonnegative integers, respectively, while we denote by $\overline{\rr}$ the set $\rr \cup \{\infty\}$. For nonnegative integers $m$ and $n$, let $\llbracket m, n \rrbracket$ be the set of integers between $m$ and $n$, i.e.,
	\[
		\llbracket m,n \rrbracket \coloneqq \{k \in\nn_0 \mid m \leq k \leq n\}.
	\]
	Given a subset $S$ of the rational numbers, we let $S_{\geq t}$ denote the set of nonnegative elements of $S$ that are greater than or equal to $t$. In the same way we define $S_{>t}$ and $S_{<t}$. For a positive rational number $q$, the relatively prime positive integers $n$ and $d$ for which $q = n/d$ are denoted by $\mathsf{n}(q)$ and $\mathsf{d}(q)$, respectively. 
	
	A monoid $M$ is \emph{reduced} if the only invertible element of $M$ is the identity. From now on we assume that all monoids here are commutative, cancellative, and reduced. Let $M$ be a monoid, which is written additively, and set $M^{\bullet} \coloneqq M\setminus\{0\}$. An element $x \in M^{\bullet}$ is an \emph{atom} provided that $x$ cannot be expressed as the sum of two elements of $M^{\bullet}$. We let $\mathcal{A}(M)$ represent the set of atoms of $M$. In addition, we say that an atom $a' \in \mathcal{A}(M)$ is \emph{stable} if the set $\{a \in \mathcal{A}(M) \mid \mathsf{n}(a) = \mathsf{n}(a')\}$ has infinite cardinality. Now for a subset $S \subseteq M$, we denote by $\langle S \rangle$ the minimal submonoid of $M$ including $S$, and if $M = \langle S \rangle$ then it is said that $S$ is a \emph{generating set} of $M$. The monoid $M$ is \emph{atomic} with the proviso that $M = \langle\mathcal{A}(M)\rangle$. 
	
	\begin{definition}
		A \emph{numerical monoid} is an additive submonoid of $\nn_0$ whose complement in $\nn_0$ is finite.
	\end{definition}

	Numerical monoids are finitely generated and, therefore, atomic with finitely many atoms. Moreover, it is well known that given a subset $S$ of $\nn$, the submonoid $\langle S \rangle$ of $\nn_0$ is a numerical monoid if and only if $\gcd(S) = 1$. For an introduction to numerical monoids and for their many applications, we refer the reader to \cite{GSJCR2009} and \cite{AAPAGS2016}, respectively.
	
	\begin{definition} 
		A \emph{Puiseux monoid} is an additive submonoid of $\qq_{\geq 0}$.
	\end{definition}
	
	Puiseux monoids are natural generalizations of numerical monoids. However, Puiseux monoids have a complex atomic structure: while some of them have no atoms at all (e.g., $\langle 1/2^n \mid n \in\nn_0 \rangle$), some others have a dense set of atoms in a real interval (e.g., $\langle [1,2) \cap \qq\rangle$). Unlike numerical monoids, Puiseux monoids are not necessarily finitely generated. Readers can find a survey about the atomic properties of Puiseux monoids in \cite{CGG2019}.
	
	The \emph{factorization monoid} of $M$, denoted by $\mathsf{Z}(M)$, is the free commutative monoid on $\mathcal{A}(M)$. The elements of $\mathsf{Z}(M)$ are called \emph{factorizations}, and if $z = a_1 + \cdots + a_n$ is an element of $\mathsf{Z}(M)$ for $a_1, \ldots, a_n \in\mathcal{A}(M)$ then it is said that $|z| \coloneqq n$ is the \emph{length} of $z$. The unique monoid homomorphism $\pi\colon\mathsf{Z}(M) \to M$ satisfying that $\pi(a) = a$ for all $a \in\mathcal{A}(M)$ is called the \emph{factorization homomorphism} of $M$. For all $x \in M$, there are two important sets associated with $x$:
	\[
	\mathsf{Z}_M(x) \coloneqq \pi^{-1}(x) \subseteq \mathsf{Z}(M) \hspace{0.6 cm}\text{ and } \hspace{0.6 cm}\mathsf{L}_M(x) \coloneqq \{|z| : z \in\mathsf{Z}_M(x)\},
	\]
	which are called the \emph{set of factorizations} of $x$ and the \emph{set of lengths} of $x$, respectively; we omit subscripts when $M$ is clear from the context. In addition, the collection $\mathcal{L}(M) \coloneqq \{\mathsf{L}(x) \mid x \in M\}$ is called the \emph{system of sets of lengths} of $M$. The system of sets of lengths of Puiseux monoids was first studied in \cite{GOTTI2019L}. See~\cite{AG2016} for a survey about sets of lengths and the role they play in factorization theory. 
	
	We now introduce unions of sets of lengths and local elasticities. The elasticity of a monoid $M$ is an invariant introduced by Valenza~\cite{Valenza} in the context of algebraic number theory, and it is defined by $\rho(M) \coloneqq \sup\{\rho_M(x) \mid x \in M\}$, where $\rho_M(0) \coloneqq 1$ and $\rho_M(x) \coloneqq \sup \mathsf{L}_M(x)/\inf\mathsf{L}_M(x)$ if $x \neq 0$. The monoid $M$ has \emph{accepted} elasticity provided that there exists $x \in M$ such that $\rho(x) = \rho(M)$. The elasticity of Puiseux monoids has been studied in \cite{MGOTTI2019,GoOn2017}. Now for a positive integer $n$, we denote by $\,\mathcal{U}_n(M)$ the set of positive integers $m$ for which there exist $a_1, \ldots, a_n, a'_1, \ldots, a'_m \in \mathcal{A}(M)$ such that $a _1 + \cdots + a_n = a'_1 + \cdots + a'_m$. It is said that $\,\mathcal{U}_n(M)$ is the \emph{union of sets of lengths} of $M$ containing $n$. It is also said that $\rho_n(M) \coloneqq \sup \,\mathcal{U}_n(M)$ is the \emph{nth local elasticity} of $M$. Unions of sets of lengths were introduced in~\cite{ChWWS1990}.
	
	A factorization invariant that is closely related to the set of lengths is the \emph{set of distances} or \emph{delta set}. For a nonzero element $x \in M$ it is said that $d \in \nn$ is a \emph{distance} of $x$ on condition that $\mathsf{L}_M(x) \cap [l, l + d] = \{l, l + d\}$ for some $l \in \mathsf{L}_M(x)$. The \emph{set of distances of $x$}, denoted by $\Delta_M(x)$, is the set consisting of all the distances of $x$. In addition, the set	
	\[	
	\Delta(M) \coloneqq \bigcup_{x \in M} \Delta_M(x)
	\]   
	is called the \emph{set of distances} of $M$. Although the set of distances of numerical monoids has received some attention lately (see, for instance, \cite{BCKR06,CKDH2009}), the set of distances of Puiseux monoids does not seem to be investigated yet.
	
%	Recently, Chapman et al.~\cite{ChOnPo2020} introduced a new factorization invariant that provides information about the entire set of lengths. Let $M$ be a BFM that is not an HFM, and let $x \in M^{\bullet}$. If $|\mathsf{L}(x)| \neq 1$ then we set 	
%	\[
%		\LD_M(x) \coloneqq \frac{|\mathsf{L}(x)| - 1}{\max \mathsf{L}(x) - \min\mathsf{L}(x)}.
%	\]
%	Additionally, we set $\LD(M) \coloneqq \inf\{\LD_M(x) \mid x \in M^{\bullet}, \,|\mathsf{L}(x)| \neq 1\}$. 
	
	\section{Set of Lengths and Elasticity}
	
	An atomic Puiseux monoid $M$ can be represented as an increasing union of copies of numerical monoids: the monoid $M$ contains a minimal set of generators, namely $\mathcal{A}(M)$, by \cite[Proposition 1.1.7]{AGHK06}. Consequently, given an ordering $a_1, a_2, \ldots$ of the elements of $\mathcal{A}(M)$, we have the sequence $(N_i)_{i \geq 1}$ with $N_i = \langle a_1, \ldots, a_i \rangle$ for all $i \in \nn$. Clearly, $M = \bigcup_{i \geq 1} N_i$ and $N_i$ is isomorphic to a numerical monoid for each $i \in \nn$ by \cite[Theorem 4.2]{GOTTI2018}. This representation has been used to manufacture Puiseux monoids satisfying certain properties. Consider the following examples.
	
	\begin{example} \label{ex: bifurcus Puiseux monoid}
		In \cite[Section 6]{GoOn2017} the authors constructed a bifurcus Puiseux monoid, that is, a Puiseux monoid $M$ satisfying that $2\in\mathsf{L}(x)$ for all $x \in M^{\bullet}\setminus\mathcal{A}(M)$. To achieve this, take a collection of prime numbers $\{p_{j,n} \mid j,\!n \geq 1\}$ such that $p_{j,n} \geq \max(13, 2^j)$ for all $j,n \in\nn$ and, recursively, define an increasing sequence of finitely generated Puiseux monoids in the following manner: take $N_0 = \langle 1/2, 1/3 \rangle$, and assuming that $N_{j - 1}$ was already defined for some $j \in \nn$, let $x_{j,1}, x_{j,2}, \ldots$ be the elements of $N_{j - 1}$ with no length $2$ factorization. Then take 
		\[
		N_j = N_{j - 1} + \left\langle\frac{x_{j,n}}{2} - \frac{1}{p_{j,n}}, \,\frac{x_{j,n}}{2} + \frac{1}{p_{j,n}} \,\,\bigg| \,\,n \geq 1 \right\rangle\!.
		\]
		Observe that $N_j$ provides a length $2$ factorization for the elements of $N_{j - 1}$ that did not have one before. Now take $M = \bigcup_{i \geq 0} N_i$. The monoid $M$ is bifurcus; the reader can check the details of the proof in \cite[Theorem 6.2]{GoOn2017}. One of the key features of this construction is that $\mathcal{A}(N_i) \subseteq \mathcal{A}(N_{i + 1})$ for every $i \in\nn_0$.
	\end{example}
	
	\begin{example} \label{ex: infinitely many Puiseux monoids with no finite local elasticities}
		In \cite{MGOTTI2019} the author proved that there exists a Puiseux monoid without $0$ as a limit point that has no finite local elasticities. With this purpose, she pieces together a Puiseux monoid $M$ by creating a strictly increasing sequence of finite subsets of positive rationals $(A_i)_{i \geq 1}$ satisfying the following three conditions:
		\begin{itemize}
			\item $\mathsf{d}(A_i)$ consists of odd prime numbers,
			\item $\mathsf{d}(\max A_i) = \max \mathsf{d}(A_i)$, and
			\item $A_i$ minimally generates the Puiseux monoid $N_i = \langle A_i \rangle$. 
		\end{itemize}
		Then the author takes $M \!= \!\bigcup_{i \geq 1} N_i$, where $\mathcal{A}(N_i) \subseteq \mathcal{A}(N_{i + 1}) \subseteq \mathcal{A}(M)$ and prove that $(\rho_2(N_i))_{i \geq 1}$ is an increasing sequence that does not stabilize. Since $\mathcal{A}(N_i)\subseteq\mathcal{A}(M)$ for each $i \in\nn$, it follows that $\rho_2(M) = \infty$. For details see \cite[Proposition 3.6]{MGOTTI2019}.
	\end{example}

	This representation of Puiseux monoids can help us not only to provide sophisticated examples but also to study some factorization invariants in these monoids. 

	\begin{definition} \label{def: numerical approximation}
		Let $(M_i)_{i \geq 1}$ be an increasing sequence of atomic Puiseux monoids. We say that $(M_i)_{i \geq 1}$ is an \emph{approximation} of the Puiseux monoid $M = \bigcup_{i \geq 1} M_i$ provided that $\mathcal{A}(M_i) \subseteq \mathcal{A}(M_{i + 1})$ for each $i \in\nn$. If $M_i$ is finitely generated for every $i \in\nn$ then we call $(M_i)_{i \geq 1}$ a \emph{numerical approximation} of $M$.
	\end{definition}
	
	\begin{remark} \label{rem: atoms of a PM are the increasing union of the atoms of the numerical semigroups forming a numerical approximation}
		\textup{Given an approximation $(M_i)_{i \geq 1}$ of a Puiseux monoid $M$, it is not hard to see that $M$ is atomic	 with $\mathcal{A}(M) = \bigcup_{i \geq 1} \mathcal{A}(M_i)$.} 
	\end{remark}

	We prove that, given an approximation of a Puiseux monoid, we can compute its sets of lengths and related factorization invariants by ``passing to the limit" in a sense that will become clear soon. Using this approach we can provide alternative proofs to some known results about the sets of lengths of Puiseux monoids.
	\begin{theorem} \label{theorem: main idea}
		Let $M$ be a Puiseux monoid with an approximation $(M_i)_{i \geq 1}$, and let $x$ be an element of $M$. Then, for some $j \in \nn$, the following statements hold:
		\begin{enumerate}
			\item $\mathsf{Z}_M(x) = \bigcup_{i \geq j} \mathsf{Z}_{M_i}(x)$ and $\mathsf{Z}(M) = \bigcup_{i \geq 1}\mathsf{Z}(M_i)$.\vspace{1pt}
			\item $\mathsf{L}_M(x) = \bigcup_{i \geq j}\mathsf{L}_{M_i}(x)$.\vspace{1pt}
			\item $\rho_M(x) = \lim_{i} \rho_{M_{i + j}}(x)$ and $\rho(M) = \lim_{i} \rho(M_i)$.\vspace{1pt}
			\item $\rho_m(M) = \lim_{i} \rho_m(M_i)$ for each $m \in \nn$.
		\end{enumerate}
	\end{theorem}

	\begin{proof}
		Let $j, r, s \in \nn$ such that $x \in M_j$ and $j \leq r \leq s$. Since $\mathcal{A}(M_r) \subseteq \mathcal{A}(M_{s})$, the inclusion $\mathsf{Z}_{M_r}(x) \subseteq \mathsf{Z}_{M_s}(x)$ holds. Now if $z\in \mathsf{Z}_{M_i}(x)$ for some $i \in \nn$ then $z \in\mathsf{Z}_M(x)$ by Remark \ref{rem: atoms of a PM are the increasing union of the atoms of the numerical semigroups forming a numerical approximation}. Conversely, if $z = a_1 + \cdots + a_n \in\mathsf{Z}_M(x)$ with $a_1, \ldots, a_n \in \mathcal{A}(M)$ then there exists $k \in\nn_{\geq j}$ such that $a_i \in \mathcal{A}(M_k)$ for each $i \in\llbracket 1,n \rrbracket$. Consequently, $z \in\mathsf{Z}_{M_k}(x)$. Hence $\mathsf{Z}_M(x) = \bigcup_{i \geq j} \mathsf{Z}_{M_i}(x)$. For all $y \in M$, let $j(y) \in \nn$ such that $y \in M_{j(y)}$. Thus,
		\[
			\mathsf{Z}(M) = \bigcup_{y \in M} \mathsf{Z}_M(y) = \bigcup_{y \in M} \bigcup_{i \geq j(y)}\mathsf{Z}_{M_i}(y) = \bigcup_{i \geq 1}\mathsf{Z}(M_i),
		\]
		from which $(1)$ follows. It is easy to see that $(2)$ readily follows from $(1)$.
		
		If $x = 0$ then the first part of $(3)$ clearly follows, so there is no loss in assuming that $x \neq 0$. Since $\mathsf{L}_{M_r}(x) \subseteq \mathsf{L}_{M_s}(x) \subseteq \mathsf{L}_M(x)$, the inequalities $\rho_{M_r}(x) \leq \rho_{M_s}(x) \leq \rho_M(x)$ hold, which implies that $\lim_{i} \rho_{M_{i + j}}(x)$ exists (in $\overline{\rr}$) and $\lim_{i} \rho_{M_{i + j}}(x) \leq \rho_M(x)$. For the reverse inequality, note that if $\mathsf{L}_M(x)$ is unbounded then $\rho_M(x) = \infty$. In this case, for each $n \in\nn$, there exists $z = a_1 + \cdots + a_l \in\mathsf{Z}_M(x)$ with $a_1, \ldots, a_l \in \mathcal{A}(M)$ satisfying that $l > n$. By virtue of $(2)$, there exists $k \in\nn_{\geq j}$ such that $l \in \mathsf{L}_{M_k}(x)$. Since $\mathsf{L}_{M_{i + j}}(x) \subseteq \mathsf{L}_{M_{i + j + 1}}(x)$ for each $i \in \nn$, we have $\lim_i \rho_{M_{i + j}}(x) = \infty$. On the other hand, if $\mathsf{L}_M(x)$ is bounded then, for some $h \in\nn_{\geq j}$, we have
		\begin{equation*}
				\rho_M(x) = \frac{\sup \mathsf{L}_M(x)}{\inf \mathsf{L}_M(x)} = \frac{\sup \bigcup_{i \geq j}\mathsf{L}_{M_i}(x)}{\inf \bigcup_{i \geq j}\mathsf{L}_{M_i}(x)} 
			 = \frac{\max \mathsf{L}_{M_h}(x)}{\min \mathsf{L}_{M_h}(x)}
			 = \rho_{M_h}(x) \leq \lim_{i \to \infty} \rho_{M_{i + j}}(x).
		\end{equation*}
		Next we prove that $\rho(M) = \lim_i \rho(M_i)$. We already established that, for each $i \in \nn$, the inequality $\rho_{M_i}(y) \leq \rho_{M_{i + 1}}(y)$ holds for all $y \in M_i$. Consequently, $\rho(M_i) \leq \rho(M_{i + 1})$ for each $i \in \nn$ which, in turn, implies that $\lim_{i} \rho(M_i)$ exists (in $\overline{\rr}$). By definition, $\rho(M) \geq \rho_M(y)$ for all $y \in M$. Now fix $j \in \nn$, and let $y' \in M_j$. Since $\rho_M(y') \geq \rho_{M_j}(y')$, the inequality $\rho(M) \geq \rho_{M_j}(y')$ holds for all $y' \in M_j$, which implies that $\rho(M) \geq \rho(M_j)$. This, in turn, implies that $\rho(M) \geq \lim_i \rho(M_i)$. To prove the reverse inequality, observe that, for all $y \in M$, we have $\rho_M(y) = \lim_i \rho_{M_{i + j(y)}}(y) \leq \lim_i \rho(M_i)$. This implies that $\rho(M) \leq \lim_i \rho(M_i)$, and $(3)$ holds.
		
		For all $i \in \nn$, the inclusions $\mathcal{U}_m(M_i) \subseteq \mathcal{U}_m(M_{i + 1}) \subseteq \mathcal{U}_m(M)$ hold. Consequently, $\sup\,\mathcal{U}_m(M_i) \leq \sup\,\mathcal{U}_m(M_{i + 1}) \leq\sup\,\mathcal{U}_m(M)$ which, in turn, implies that $\lim_{i} \rho_m(M_i)$ exists (in $\overline{\rr}$) and $\lim_{i} \rho_m(M_i) \leq \rho_m(M)$. Now if $\mathcal{U}_m(M)$ is unbounded then for each $N \in\nn$ there exist $x \in M$ and $z, z' \in\mathsf{Z}_M(x)$ such that $|z| > N$ and $|z'| = m$. Since there exists $j \in \nn$ such that $z, z' \in \mathsf{Z}_{M_j}(x)$, the inequality $\rho_m(M_j) > N$ holds. This implies that $\lim_{i} \rho_m(M_i) = \infty$. Then there is no loss in assuming that $k \coloneqq \sup\,\mathcal{U}_m(M)$ is a positive integer. Let $x \in M$ such that $|z| = k$ and $|z'| = m$ for some $z,z' \in\mathsf{Z}_M(x)$. Since $z,z' \in \mathsf{Z}_{M_j}(x)$ for some $j \in\nn$, our argument follows.
	\end{proof}

%	\begin{cor} \label{cor: length density}
%		With notation as in Theorem~\ref{theorem: main idea}, if $M$ is a BFM such that $M \neq \nn_0$ and $x \in M^{\bullet}$ with $|\mathsf{L}(x)| \neq 1$ then $\LD_M(x) = \lim_{i} \LD_{M_{i + j}}(x)$ and $\LD(M) = \lim_i \LD(M_i)$.
%	\end{cor}
%
%	\begin{proof}
%		There is no loss in assuming that $M_i \neq \nn_0$ for each $i \in \nn$; consequently, the expression $\LD(M_i)$ is well defined by \cite[Corollary 1.3.3]{AGHK06} and \cite[Proposition 4.22]{CGG2019}. For some $j \in \nn$, we have $|\mathsf{L}_{M_j}(x)| \neq 1$ and we already established that, for every $l \in \mathsf{L}_M(x)$, there exists $k \in \nn$ such that $l \in \mathsf{L}_{M_{k + j}}(x)$. Since $M$ is a BFM, we may assume $\mathsf{L}_{M}(x) = \mathsf{L}_{M_{k + j}}(x)$. Hence we just need to prove that if $M$ has no accepted length density then $\LD(M) = \lim_i \LD(M_i)$.
%	\end{proof}

	\begin{cor} \cite[Theorem 3.2]{GoOn2017} \label{cor: formula for the elasticity of Puiseux monoids}
		Let $M$ be an atomic Puiseux monoid. If $0$ is a limit point of $M^{\bullet}$ then $\rho(M) = \infty$. Otherwise, $ \rho(M) = \frac{\sup\mathcal{A}(M)}{\inf\mathcal{A}(M)}$.
	\end{cor}
	\begin{proof}
		Let $(N_i)_{i \geq 1}$ be a numerical approximation of $M$. If $0$ is a limit point of $M^{\bullet}$ then, for each $n \in \nn$, there exists $j \in \nn$ such that $\rho(N_j) > n$ by \cite[Theorem 2.1]{ChHoMoo2006}, which implies that $\lim_{i} \rho(N_i) = \infty$ since $(\rho(N_i))_{i \geq 1}$ is nondecreasing. Now if $0$ is not a limit point of $M^{\bullet}$ then
		\begin{equation*}
				\rho(M) = \lim_{i \to \infty} \,\rho(N_i)
				= \lim_{i \to \infty} \, \frac{\max \mathcal{A}(N_i)}{\min \mathcal{A}(N_i)}
				= \frac{\sup\mathcal{A}(M)}{\inf\mathcal{A}(M)},
		\end{equation*}
		where the second equality follows from \cite[Theorem 2.1]{ChHoMoo2006}.
	\end{proof}
	\begin{cor} \cite[Theorem 3.4]{GoOn2017}
		Let $M$ be an atomic Puiseux monoid satisfying that $\rho(M) < \infty$. Then the elasticity of $M$ is accepted if and only if $\mathcal{A}(M)$ has both a maximum and a minimum.
	\end{cor}
	\begin{proof}
		Let $(N_i)_{i \geq 1}$ be a numerical approximation of $M$. To tackle the direct implication, note that for some $x \in M$, $j \in \nn$, and $L,l \in \mathsf{L}_M(x)$ we have
		\[
			\frac{\sup \mathcal{A}(M)}{\inf \mathcal{A}(M)} = \rho(M) = \rho_M(x) = \frac{L}{l} = \rho_{N_j}(x) = \frac{\max\mathcal{A}(N_j)}{\min\mathcal{A}(N_j)},
		\]
		where the last equality follow from \cite[Theorem 2.1]{ChHoMoo2006}. The reverse implication follows from \cite[Theorem 3.1.4]{AGHK06} and the fact that, for some $j \in \nn$, the monoid $N_j$ contains the minimum and maximum of $\mathcal{A}(M)$.
	\end{proof}
	\begin{cor} \cite[Proposition 3.1]{MGOTTI2019}
		Let $M$ be an atomic Puiseux monoid. If $M$ contains a stable atom $a \in \mathcal{A}(M)$ then $\rho_k(M)$ is infinite for all sufficiently large $k$.
	\end{cor}
	\begin{proof}
		Let $(N_i)_{i \geq 1}$ be a numerical approximation of $M$, and suppose without loss of generality that $a \in N_1$. For each $j \in \nn$, there exists $k \in\nn$ such that the inequality $\rho_{\mathsf{d}(a)}(N_{j + k}) > \rho_{\mathsf{d}(a)}(N_j)$ holds since $N_j$ is finitely generated. Therefore, $\lim_{i} \rho_{\mathsf{d}(a)}(N_i) = \infty$. By Theorem~\ref{theorem: main idea}, we have $\rho_{\mathsf{d}(a)}(M) = \infty$. Our argument follows after \cite[Proposition 1.4.2]{AGHK06}.
	\end{proof}

	\section{Set of Distances and Length Density}
	
	It is straightforward to construct a Puiseux monoid $M$ with an approximation $(M_i)_{i \geq 1}$ such that $\Delta(M) \neq \bigcup_{i \geq 1} \Delta(M_i)$. Consequently, the approach we used in Theorem~\ref{theorem: main idea} to compute invariants like the set of lengths is not going to work for the set of distances. However, using limits of sets we can obtain a result similar to Theorem~\ref{theorem: main idea}.
	
	\begin{definition} \label{def: limit of sets}
		Let $(S_i)_{i \geq 1}$ be a sequence of sets, and let $\liminf_{i} S_i$ and $\limsup_{i} S_i$ be the sets
		\[
		\liminf_{i \to \infty} S_i \coloneqq \bigcup_{i \geq 1} \bigcap_{j \geq i} S_j \hspace{0.6cm} \text{ and } \hspace{0.6cm} \limsup_{i \to \infty} S_i \coloneqq \bigcap_{i \geq 1} \bigcup_{j \geq i} S_i.
		\]
		We say that $\lim_{i} S_i$ exists and is equal to $\liminf_{i} S_i$ provided that $\liminf_{i} S_i = \limsup_{i} S_i$.
	\end{definition}

	Observe that Definition~\ref{def: limit of sets} is consistent with the notation used in Theorem~\ref{theorem: main idea} since if $(S_i)_{i \geq 1}$ is an increasing sequence of sets then $\lim_{i} S_i = \bigcup_{i \geq 1} S_i$ as the reader can easily prove.
	
	\begin{proposition}\label{prop: approximating sets of distances}
		Let $M$ be a Puiseux monoid with an approximation $(M_i)_{i \geq 1}$, and let $x$ be an element of $M$. Then $\Delta_M(x) \subseteq \liminf_{i} \Delta_{M_i}(x)$ and $\Delta(M) \subseteq \liminf_{i} \Delta(M_i)$.
	\end{proposition}
	\begin{proof}
		Let $d \in \Delta_M(x)$. Then there exist factorizations $z, z' \in \mathsf{Z}_M(x)$ satisfying that $|z'| - |z| = d$ and $[|z|, |z'|] \cap \mathsf{L}_M(x) = \{|z|, |z'|\}$. Let $k \in \nn$ such that $z, z' \in \mathsf{Z}_{M_k}(x)$. By virtue of Theorem~\ref{theorem: main idea}, we have that $d \in \Delta_{M_h}(x)$ for all $h \in \nn_{\geq k}$ which, in turn, implies that $d \in \bigcap_{j \geq k} \Delta_{M_j}(x)$. Then $d \in \liminf_{i} \Delta_{M_i}(x)$. Finally, let $d \in \Delta(M)$. By definition, there exists $x \in M^{\bullet}$ such that $d \in \Delta_M(x)$. As we already showed, $d \in \bigcap_{j \geq k} \Delta_{M_j}(x)$ for some $k \in \nn$. Consequently, $d \in \bigcap_{j \geq k} \Delta(M_j)$, from which our result follows.
	\end{proof}

	Proposition~\ref{prop: approximating sets of distances} can be useful when analyzing the set of lengths of particular classes of atomic Puiseux monoids. Consider the following examples.
	
	\begin{example}
		Let $r \in \mathbb{Q}_{<1}$ such that the rational cyclic monoid over $r$, that is, $S_r \coloneqq \langle r^n \mid n \in \nn_0 \rangle$, is atomic. Then $\mathsf{n}(r) > 1$ by \cite[Theorem 6.2]{GG2018}. Fix $i \in \nn$, and consider the numerical monoid $N_i = \langle \mathsf{n}(r)^i, \mathsf{n}(r)^{i - 1}\mathsf{d}(r), \ldots, \mathsf{d}(r)^i \rangle$. By virtue of \cite[Corollary 20]{KOP2016}, we have $\Delta(N_i) = \{\mathsf{d}(r) - \mathsf{n}(r)\}$.  It is not hard to see that $(\mathsf{d}(r)^{-i} N_i)_{i \geq 1}$ is a numerical approximation of $S_r$. Therefore, $\Delta(S_r) \subseteq \{\mathsf{d}(r) - \mathsf{n}(r)\}$ by Proposition~\ref{prop: approximating sets of distances}. Following a similar reasoning we obtain that if $r > 1$ and $S_r$ is atomic then the inclusion $\Delta(S_r) \subseteq \{\mathsf{n}(r) - \mathsf{d}(r)\}$ holds. This result was first proved in \cite[Theorem 3.3]{ScGG2019}.
	\end{example}

	\begin{example} \label{ex: extending multicyclic monoids}
		Let $\mathcal{B}$ be a nonempty subset of $\qq_{>0} \setminus \nn$ such that for all $b, b' \in \mathcal{B}$ with $b \neq b'$ we have $\mathsf{n}(b) > 1$, $\gcd(\mathsf{d}(b), \mathsf{d}(b')) = 1$, and $|\mathsf{n}(b) - \mathsf{d}(b)| = |\mathsf{n}(b') - \mathsf{d}(b')|$. Set $M_{\mathcal{B}} \coloneqq \langle b^n \mid b \in \mathcal{B},\, n \in \nn_0 \rangle$. Now given an ordering $b_1, b_2, \ldots$ of the elements of $\mathcal{B}$, let $\mathcal{B}_i = \{b_1, \ldots, b_i\}$ and set $M_{\mathcal{B}_i} \coloneqq \langle b^n \mid b \in \mathcal{B}_i, \,n \in \nn_0 \rangle$ for each $i \in \nn$. The sequence $(M_{\mathcal{B}_i})_{i \geq 1}$ is an approximation of $M_{\mathcal{B}}$ by \cite[Proposition 3.5]{HP2020}. Moreover, for each $i \in \nn$, $\Delta(M_{\mathcal{B}_i}) = \{|\mathsf{n}(b_1) - \mathsf{d}(b_1)|\}$ by \cite[Theorem 4.9]{HP2020}. Therefore, $\Delta(M_{\mathcal{B}}) \subseteq \{|\mathsf{n}(b_1) - \mathsf{d}(b_1)|\}$ by Proposition~\ref{prop: approximating sets of distances}.
	\end{example}

	\begin{remark}
		\textup{Example~\ref{ex: extending multicyclic monoids} extends part~$(2)$ of \cite[Theorem 4.9]{HP2020} to a larger class of Puiseux monoids.}
	\end{remark}

	The next example shows that, in general, $\Delta(M) \neq \liminf_{i} \Delta(M_i)$.
	
	\begin{example} \label{ex: delta set cannot be approximated}
		Consider the rational cyclic monoid $S_r$ with $r \in \qq_{>1} \setminus \nn$. For each $i \in \nn$, set 
		\[
			M_i \coloneqq \left\langle \left\{r^{2k} \mid k \in \nn_0\right\} \cup \left\{r^{2j - 1} \mid j \in \llbracket 1,i \rrbracket\right\} \right\rangle.
		\]
		It is not hard to prove that $(M_i)_{i \geq 1}$ is an approximation of $S_r$. Now fix $i \in \nn$, and let $x_i = \mathsf{n}(r)^2r^{2i} \in M_i$. Clearly, $z = \mathsf{n}(r)^2r^{2i}$ and $z' = \mathsf{d}(r)^2r^{2i + 2}$ are two factorizations of $x_i$ in $M_i$.
		\begin{claim} \label{claim: unique factorization of minimum length}
			$z' = \mathsf{d}(r)^2r^{2i + 2} \in \mathsf{Z}_{M_i}(x_i)$ is the factorization of minimum length of $x_i$ in $M_i$.
		\end{claim}
		\begin{proof}
			Let $z'' = \sum_{k = 0}^{n} c_kr^{s_k} \in \mathsf{Z}_{M_i}(x_i)$ with coefficients $c_0, \ldots, c_n \in \nn_0$ and exponents $s_0, \ldots, s_n \in \{2k \mid k \in \nn_0\} \cup \{2j - 1 \mid j \in \llbracket 1,i \rrbracket\}$, and assume by contradiction that $z''$ is a factorization of minimum length of $x_i$ in $M_i$ satisfying that $z'' \neq z'$. There is no loss in assuming that $s_l < s_r$ for $l < r$, $[r^{s_l}, r^{s_{l + 1}}] \cap \mathcal{A}(M_i) = \{r^{s_l}, r^{s_{l + 1}}\}$ for all $l \in \llbracket 0, n - 1 \rrbracket$ and $s_t = 2i + 2$ for some $t \in \llbracket 0,n \rrbracket$. Note that $c_k < \mathsf{n}(r)^{s_{k + 1} - s_k}$ for each $k \in \llbracket 0,n \rrbracket$; otherwise, using the transformation $\mathsf{n}(r)^{s_{k + 1} - s_k}r^{s_k} = \mathsf{d}(r)^{s_{k + 1} - s_k}r^{s_{k + 1}}$ we can generate a new factorization $z^* \in \mathsf{Z}_{M_i}(x_i)$ such that $|z^*| < |z''|$, which is a contradiction. Now let $m$ be the smallest nonnegative integer such that $c_m \neq 0$, and consider the equation 
			\begin{equation} \label{eq: equality for contradiction}
				\sum_{k = m}^{n} c_kr^{s_k} = \mathsf{d}(r)^2r^{2i + 2}.
			\end{equation}
			If $m < t$ then after clearing denominators in Equation~\eqref{eq: equality for contradiction} we generate a contradiction with the fact that $c_m < \mathsf{n}(r)^{s_{m + 1} - s_m}$. We obtain a similar contradiction for the case where $m \geq t$ as the reader can verify. Therefore, there exists exactly one factorization of minimum length of $x_i$ in $M_i$, namely $z'$.
		\end{proof}
	
		\noindent Now let $z^* = \sum_{k = 0}^{n} c_kr^{s_k} \in \mathsf{Z}_{M_i}(x_i)$ with coefficients $c_0, \ldots, c_n \in \nn_0$ and exponents $s_0, \ldots, s_n \in \{2k \mid k \in \nn_0\} \cup \{2j - 1 \mid j \in \llbracket 1,i \rrbracket\}$. Suppose, without loss of generality, that $s_l < s_r$ for $l < r$ and $[r^{s_l}, r^{s_{l + 1}}] \cap \mathcal{A}(M_i) = \{r^{s_l}, r^{s_{l + 1}}\}$ for every $l \in \llbracket 0, n - 1 \rrbracket$. Note that in the proof of Claim~\ref{claim: unique factorization of minimum length}, we established that if $c_k < \mathsf{n}(r)^{s_{k + 1} - s_k}$ for each $k \in \llbracket 0,n \rrbracket$ then $z^* = z'$.
		
		\begin{claim}\label{claim: element in the delta set}
			If $|z^*| < |z| = \mathsf{n}(r)^2$ then $z^* = z'$.
		\end{claim}
		
		\begin{proof}
			If $c_k < \mathsf{n}(r)^{s_{k + 1} - s_k}$ for each $k \in \llbracket 0,n \rrbracket$ then we are done by our previous observation. By contradiction, assume that $z^* \neq z'$. Using the transformation
			\begin{equation}\label{eq: transformation}
				\mathsf{n}(r)^{s_{k + 1} - s_k}r^{s_k} = \mathsf{d}(r)^{s_{k + 1} - s_k}r^{s_{k + 1}}
			\end{equation}
			 we can generate from $z^*$ a new factorization $z_1 \in \mathsf{Z}_{M_i}(x_i)$ such that $|z_1| < |z^*|$. Then either $z_1 = z'$ or we can again apply the transformation~\eqref{eq: transformation} to obtain a new factorization $z_2 \in \mathsf{Z}_{M_i}(x_i)$ such that $|z_2| < |z_1|$, and so on. This procedure stops since there is no strictly decreasing sequence of nonnegative integers. Then there exist factorizations $z^* = z_0, z_1, \ldots, z_m = z'$ such that $|z_j| > |z_{j + 1}|$ for every $j \in \llbracket 0, m - 1 \rrbracket$. It should be noted that the transformation~\eqref{eq: transformation} increases the exponent of $r$, which means that $c_k = 0$ for all $s_k > 2i + 2$, where $k \in \llbracket 0,n \rrbracket$.
%			  Given that $r^{2i + 1} \not\in M_i$ and that the transformation~\eqref{eq: transformation} introduces coefficients of the form $\mathsf{d}(r)^{s_{k + 1} - s_k}$, 
%			 the inequality $s_k < 2i + 2$ holds for all $k \in \llbracket 0,n \rrbracket$.
			 This implies that at some point in the aforementioned procedure we applied the transformation $\mathsf{n}(r)^2r^{2i} = \mathsf{d}(r)^2 r^{2i + 2}$, but this contradicts that $|z^*| < |z| = \mathsf{n}(r)^2$. 
		\end{proof}
	\noindent Because of Claim~\ref{claim: element in the delta set}, $\mathsf{n}(r)^2 - \mathsf{d}(r)^2 \in \Delta(M_i)$ for every $i \in \nn$. Consequently, we have that $\mathsf{n}(r)^2 - \mathsf{d}(r)^2 \in \liminf_{i} \Delta(M_i)$. However, we know that $\Delta(S_r) = \{\mathsf{n}(r) - \mathsf{d}(r)\}$ by \cite[Corollary 3.4]{ScGG2019}. Therefore, $\Delta(S_r) \neq \liminf_{i} \Delta(M_i)$.
	\end{example}
	
	Example~\ref{ex: delta set cannot be approximated} is rather complicated, but notice that an approximation $(M_i)_{i \geq 1}$ of a Puiseux monoid $M$ satisfying that $\Delta(M) \neq \liminf_{i} \Delta(M_i)$ does never stabilize, which means that, in particular, $M$ is not finitely generated. On the other hand, rational cyclic monoids are perhaps the non-finitely generated Puiseux monoids with more tractable factorization invariants (see \cite{ScGG2019}).

\section*{Acknowledgments}
	
	 The author wants to thank Felix Gotti not only for providing the initial questions that motivated this paper but also for his guidance during its preparation. While working on this article, the author was supported by the University of Florida Mathematics Department Fellowship.

\end{document}